\documentclass[12pt]{amsart}

\usepackage[english]{babel}
\usepackage[T1]{fontenc}
\usepackage[utf8]{inputenc}

\usepackage{a4wide}

\usepackage{amssymb}
\usepackage{amsthm}
\usepackage{amsfonts}
\usepackage{mathtools}
\usepackage{color}
\usepackage{graphicx}

\title[Weak compactness for 4-th order systems]{Weak compactness of solutions for fourth order elliptic systems with critical growth}

\author{Pawe\l{} Goldstein}
\address{Pawe\l{} Goldstein\hfill$\phantom{a}$ \linebreak\hspace*{9pt} Institute of Mathematics, University of Warsaw, Banacha 2, 02--097 Warszawa, Polska}
\email{P.Goldstein@mimuw.edu.pl}
\thanks{All authors partially supported by MNiSW Grant no N N201 397737, \emph{Nonlinear partial differential equations: geometric and variational problems}}

\author{Pawe\l{} Strzelecki}
\address{Pawe\l{} Strzelecki\hfill$\phantom{a}$ \linebreak\hspace*{9pt} Institute of Mathematics, University of Warsaw, Banacha 2, 02--097 Warszawa, Polska}
\email{P.Strzelecki@mimuw.edu.pl}

\author{Anna Zatorska-Goldstein}
\address{Anna Zatorska-Goldstein\hfill$\phantom{a}$ \linebreak\hspace*{9pt} Institute of Applied Mathematics and Mechanics, University of Warsaw, Banacha 2,\linebreak\hspace*{9pt} 02--097~Warszawa, Polska}
\email{A.Zatorska-Goldstein@mimuw.edu.pl}

\date{\today}

\newcommand{\bb}{\mathbb{B}}
\newcommand{\R}{\mathbb{R}}
\newcommand{\cC}{\mathcal{C}}
\newcommand{\eps}{\varepsilon}
\newcommand{\tD}{\tilde{D}}
\newcommand{\tE}{\tilde{E}}

\newcommand{\dv}{\mathrm{div}}

\newcommand{\les}{\lesssim}
\newcommand{\rhu}{\rightharpoonup}
\def\charfn_#1{{\raise1.2pt\hbox{$\chi  %
_{\kern-1pt\lower4pt\hbox{{$\scriptstyle#1$}}}$}}}  %

\theoremstyle{plain}
\newtheorem{lemma}{\bf Lemma}[section]
\newtheorem{theorem}[lemma]{\bf Theorem}

\newtheorem{corollary}[lemma]{\bf Corollary}

\theoremstyle{definition}

\newtheorem{example}[lemma]{\bf Example}

\theoremstyle{remark}
\newtheorem{remark}[lemma]{\bf Remark}
\newtheorem*{remark*}{\bf Remark}

\providecommand{\MR}[1]{}

 \numberwithin{equation}{section}

\begin{document}

\begin{abstract}
We consider a class of fourth order elliptic systems which include the Euler--Lagrange equations of biharmonic mappings in dimension $4$ and we prove that weak limit of weak solutions to such systems is again a weak solution to a limit system.
\end{abstract}

\maketitle
\providecommand{\MR}[1]{}

\section{Introduction}

\noindent
In this paper we consider fourth order elliptic systems of equations of the form
\begin{equation} \label{system_type}
\Delta^2 u=\Delta (D\cdot\nabla u)+\dv (E\cdot \nabla u)+G\cdot\nabla u + \Delta \Omega\cdot \nabla u
\end{equation}
for an unknown map $u\colon\bb^4\subset\R^4\to\R^m$,
i.e., in components,
\begin{equation}
\Delta^2 u^i=\Delta \Big( D^{ij}_\alpha \partial_\alpha u^j \Big)
    + \partial_\alpha \Big( E^{ij}_{\alpha \beta} \partial_\beta u^j \Big)
    + G^{ij}_\alpha \partial_\alpha u^j + \Delta (\Omega^{ij}_\alpha) \partial_\alpha u^j
\end{equation}
where $\alpha, \beta = 1,2,3,4$ and $i,j = 1,\ldots,m$. The coefficient functions $D,E,G,\Omega$ are assumed to satisfy
\begin{equation}
	\label{c spaces}
	D\in W^{1,2}, \qquad E\in L^2, \qquad G\in L^{4/3}, \qquad
	\Omega\in W^{1,2}(\bb^4; \text{\textit{so}}(m) \otimes \Lambda^1 \R^4)\, .
\end{equation}
We study compactness of the space of solutions in the weak sequential topology of the Sobolev space $W^{2,2}$.

Let us note immediately that under the above assumptions $G\cdot \nabla u$ is just in $L^1$, as $\nabla u\in W^{1,2}\subset L^4$ by Sobolev imbedding in dimension $4$. Thus, \eqref{system_type} is critical. However, the whole point is that we allow  $D,E,G,\Omega$ to depend nonlinearly on $u$. The class of systems we consider contains, in particular, the Euler--Lagrange equations of biharmonic maps from domains in $\R^4$ into compact Riemannian manifolds. Our approach relies in a crucial way on the antisymmetry of the 1-form $\Omega$ and on the use of nonlinear counterparts of the Hodge decomposition, originating in gauge theory. This key idea is due to T.~Rivi\`{e}re; it has been first used in his pioneering paper \cite{Riviere} on conformally invariant second order elliptic systems in the plane, with harmonic maps from planar domains into compact Riemannian manifolds serving as the crucial example. Later on,   Rivi\`{e}re--Struwe \cite{RS}, Lamm-- Rivi\`{e}re \cite{LammRiviere}, and Struwe \cite{Struwe} extended this approach to stationary harmonic maps in higher dimensions and to biharmonic maps.

Let us also note that \eqref{system_type} in full generality is wider than the class of biharmonic maps. It can happen in dimension $4$ that a solution of \eqref{system_type} is continuous, even $C^{1,\lambda}$, but still not $C^2$. We shall comment on that later on; let us now state the main result.

\begin{theorem} \label{main theorem}
Suppose $(u_k)$ is a sequence in $W^{2,2}(\bb^4,\R^m)$ of weak solutions to the system 
\begin{equation}\label{system_k}
\Delta^2 u_k=\Delta (D_k\cdot\nabla u_k)+\dv (E_k\cdot \nabla u_k)+G_k\cdot\nabla u_k + \Delta \Omega_k\cdot \nabla u_k \quad \text{in $\bb^4$}.
\end{equation}
Suppose $u_k \rhu u$ in $W^{2,2}(\bb^4,\R^m)$. If the coefficients $D_k$, $E_k$, $G_k$, $\Omega_k$ 
are weakly convergent in their respective Sobolev spaces, i.e.
\begin{gather}
D_k \xrightharpoonup{W^{1,2}} D,
\qquad  \Omega_k \xrightharpoonup{W^{1,2}} \Omega,\qquad
E_k \xrightharpoonup{L^{2}} E,
\qquad G_k \xrightharpoonup{L^{4/3}} G,
\label{C conv 2}
\end{gather}
then $u$ is a weak solution to the limit system
\begin{equation} \label{system}
\Delta^2 u=\Delta (D\cdot\nabla u)+\dv (E\cdot \nabla u)+G\cdot\nabla u + \Delta \Omega\cdot \nabla u \quad \text{in $\bb^4$}.
\end{equation}
\end{theorem}

Related compactness results for biharmonic maps, along with an energy identity accounting for the possible `energy loss' under the limit passage have been obtained by Hornung and Moser \cite{HornungMoser} and Laurain and  Rivi\`{e}re \cite{LaurainR}. (For second order elliptic systems $\Delta u_k=\Omega_k\cdot \nabla u_k$, Sharp and Topping \cite{SharpTopping} provide a compactness theorem under an assumption that rules out the concentration of the energy $\int |\nabla u_k|^2$ but allows for concentration of $\int |\nabla^2 u_k|$.)

One of our main points is that the proof in this paper is different from \cite{HornungMoser} and \cite{LaurainR}:  contrary to these two papers, we do not rely at all either on continuity of solutions (or their first and second order derivatives) or on any other improved regularity properties (like higher integrability of $\nabla^2 u$), working all the time just in $W^{2,2}$ and using  the imbedding $W^{2,2}\subset W^{1,4}$.\;{}\footnote{In Section~\ref{example} we provide an example showing that \eqref{system_type} admits weak solutions that are $C^1$ but not $C^2$.} Besides K.~Uhlenbeck's Theorem~\ref{Uhlenbeck}, the main tool is the concentration--compactness method of P.-L. Lions. The combination of the two allows us, very roughly speaking, to reduce the complexity of \eqref{system_k}--\eqref{system} to the case
\[
\Delta^2 u_k = G_k\cdot \nabla u_k,  \qquad u_k\rhu u \quad\text{in \;} W^{2,2},
\]
where $G_k$ is a bounded sequence in $L^{4/3}$, with $G_k\rhu G$. 
Then,  a passage to the limit (in dimension 4) can be justified by an application of Sobolev inequality, the concentration--compactness lemma \cite[Lemma~1.2]{Lions}, and a standard capacity type argument.

\begin{remark} Conditions \eqref{c spaces} imposed on the coefficients allow for nonlinear dependence on $u$. Roughly speaking, the class \eqref{system_type} contains systems of the form $\Delta^2 u= R(u,\nabla u, \nabla^2u)$ where the right hand side $R$ depends polynomially on (the entries of) $\nabla u$ and $\nabla^2 u$, with coefficients that are smooth and bounded in $u$, so that
	\[
	|R(u,\nabla u,\nabla^2 u)|\lesssim |\nabla^2 u|^2 +  |\nabla u|^4 + 1\, .
	\]
The point is that the terms of $R$ depending quadratically on $\nabla^2 u$ \emph{need to have some structure\/}, whereas all the other terms are allowed to be arbitrary. A model case looks as follows (summation over repeated indices is understood):
\begin{gather}
D=D(u, \nabla u) \in W^{1,2}, \qquad D^{ij}_\alpha=d^{ij}_{l}(u)\partial_\alpha u^l\, ;\label{ex D}\\
E=E(u,\nabla u, \nabla^2 u) \in L^2, \qquad E^{ij}_{\alpha\beta}=e_{l}^{ij,(1)}(u)\partial_\alpha(\partial_\beta u^l) + e_{ls}^{ij,(2)}(u)\partial_\alpha u^l\partial_\beta u^s\, ;\\
G=G(u, \nabla u,\nabla^2 u) \in L^{4/3}, \qquad G^{ij}_\alpha = g_{lsp,\beta\gamma}^{ij}(u)\big(\partial_\beta(\partial_\gamma u^l) +\partial_\beta u^l\partial_\gamma u^s\big)\partial_\alpha u^p\, ;\label{ex F}
\end{gather}
finally, the antisymmetric term $\Omega=\Omega(u, \nabla u) \in W^{1,2}(\bb^4; so(m) \otimes \Lambda^1 \R^4)$ is given by
\begin{equation}
	\Omega^{ij}_\alpha = w_i\partial_\alpha w^j - w^j \partial_\alpha w^i = -\Omega^{ji}_\alpha\, .
	\label{ex O}
\end{equation}
where  $w=\nu\circ u$ of $u$ for some bounded smooth map $\nu\colon \R^m\to \R^m$.

All the coefficients $d,e,g$ (with various indices) in \eqref{ex D}--\eqref{ex F} are assumed to be of class $C^1 \cap W^{1,\infty}$ on $\R^m$. Struwe \cite[Section 2]{Struwe} explains that the biharmonic map equation can be written in that form, with $\nu$ being the normal to the target manifold. In that case, $d,e,g$ depend explicitly on $\nu$ and its derivatives, and the growth estimates
\begin{align}
|D| + |\Omega| &\les |\nabla u|, \nonumber\\
|E|+|\nabla D| +|\nabla \Omega|&\les |\nabla^2 u|+|\nabla u|^2,\label{Gr E}\\
|G| &\les |\nabla^2 u| |\nabla u| + |\nabla u|^3 \nonumber
\end{align}
follow from \eqref{ex D}--\eqref{ex O}.
Under these assumptions, for every weakly convergent sequence $u_k\rhu u$ in $W^{2,2}$ we have convergence \eqref{C conv 2}, with
\[
L^2\ni E_k=E(u_k,\nabla u_k,\nabla^2 u_k)\rhu E=E(u,\nabla u, \nabla^2 u)\]
in $L^p$ for $p<2$, and $L^{4/3}\ni G_k=G(u_k,\nabla u_k,\nabla^2 u_k)\rhu G= G(u,\nabla u, \nabla^2 u)$ in $L^s$ for $s<4/3$. Since $E_k$ is bounded in $L^2$ and $G_k$ is bounded in $L^{4/3}$, and we deal \emph{with a bounded domain},  it is an exercise to see that in fact  $E_k\rhu E$ and $G_k\rhu G$ also for the limiting exponents $p=2$ and $s=4/3$.

We do not need the full strength of \eqref{ex D}--\eqref{Gr E}.
\end{remark}
\medskip

\begin{remark} Our proof depends in a crucial way on Sobolev imbedding in dimension~4. It would be interesting to know what happens in higher dimensions. For example: is the convergence $u_k\to u$ in $BMO$, combined with the boundedness of $(u_k)$ in $W^{2,2}$, sufficient to guarantee (a) that $u$ solves the limiting system, (b) that $u_k\to u$ strongly in $W^{2,2}$? It is possible to check, using the sharp version of Gagliardo--Nirenberg inequality $\|\nabla w\|_{L^4}^2\lesssim \|w\|_{BMO}\|w\|_{W^{2,2}}$ (cf. Meyer--Rivi\`{e}re \cite{MeyerR}, or \cite{strzelecki-bmo}), that both answers would be positive for uniformly bounded weak solutions of the simplified system $\Delta^2 u_k=G_k\nabla u_k$, with $G_k=G(u_k,\nabla u_k, \nabla^2 u_k)$, where $G$ is of the form \eqref{ex F}. It seems plausible that convergence of, say, biharmonic maps in $BMO$ prevents bubbling and loss of energy in the limit.
\end{remark}

The rest of the paper is organized as follows. In Section~2, we recall Uhlenbeck's result in the form that is used later on.  Section~3 forms the bulk of the paper. There, we first explain the strategy of the proof in more detail, and then carry out the necessary estimates, pass to the limit and remove the singularities of the limiting system. Finally, in Section~4 we give an example showing that \eqref{system_type} can have solutions in $C^1\setminus C^2$.

\section{Uhlenbeck's result}
\label{sec:U}

We remind that in our case we consider a mapping $u$ going from a ball $\bb^4 \subset \R^4$ into $\R^m$. Below we state the theorem of Uhlenbeck in a form adjusted to the situation. Note that if $\Omega \in W^{1,2}(\bb^4)$ then $\Omega \in L^4(\bb^4)$ because of the Sobolev imbedding theorem.

\begin{theorem} \label{Uhlenbeck}
There exist a number $\eps=\eps(m)>0$ and a constant $C>0$ such that for any $\Omega \in W^{1,2}(\bb^4; so(m)\otimes \Lambda^1\R^4)$ which satisfies
\begin{equation*}
\|\nabla \Omega\|_{L^2}+\|\Omega\|^{ 2}_{L^4}\leq \eps
\end{equation*}
there exist $P\in W^{2,2}(\bb^4, SO(m))$ and $\xi \in W^{2,2}(\bb^4,so(m)\otimes \Lambda^2 \R^4)$ such that
\begin{equation}\label{eq:uhlenbeck}
(dP)\, P^{-1}+P\Omega P^{-1}=\ast d\xi \qquad \text{on \;}  \bb^4
\end{equation}
and $d(\ast \xi)=0$ on $\bb^4$, $\xi|_{\partial \bb^4}=0$.

Moreover
\begin{equation}
	\label{est:P} \|\nabla^2 P\|_{L^2}+\|\nabla P\|_{L^4}+\|\nabla^2 \xi\|_{L^2}+\|\nabla \xi\|_{L^4}\les \|\nabla\Omega\|_{L^2}+\|\Omega\|^2_{L^4}.
\end{equation}

\end{theorem}

Uhlenbeck's Theorem is, in fact, a local theorem in the sense that we can use it not only on the unit ball $\bb^4$, but on any ball, and, as long as we consider balls with uniformly bounded radii, we can choose the constant $\eps$ in an uniform way (i.e. independently on the radius of the ball). This is in accordance with the original use of this theorem to prove the existence of {\it global} Coulomb gauges on compact manifolds. Indeed, a look at the proof of Lemma 2.5 in \cite{Uhlenbeck} shows that we can choose $\eps=(2 (C_Pr+1)C_S)^{-1}$, where $C_P$ and $C_S$ are the constant in the Poincar\'e and Sobolev inequalities for the unit ball, and $r$ denotes the radius of the ball $B$. Thus $\eps=(2 (C_P\,R+1)C_S)^{-1}$ can be chosen as a uniform estimate for all balls with radius bounded by $R$.
\begin{corollary}\label{UhlCor2}
Theorem \ref{Uhlenbeck} holds for any ball $B\subset\bb^4$ in place of $\bb^4$, and the constant $\eps$ can be chosen uniformly for all such balls.
\end{corollary}

Another corollary deals with the problem of weak continuity of $P$ and $\xi$ with respect to $\Omega$.
Note that Theorem \ref{Uhlenbeck} does not claim that either of them is defined uniquely.

\begin{corollary}\label{cor:2}
Suppose $\{ \Omega_k \}$ is a sequence in $W^{1,2}(\bb^4; so(m)\otimes \Lambda^1\R^4)$. Assume
\begin{equation}\label{Omega_k wto Omega}
\Omega_k \rightharpoonup \Omega \quad \text{in} \quad W^{1,2}.
\end{equation}
Assume that $P_k$ and $\xi_k$ are chosen so that \eqref{eq:uhlenbeck} and \eqref{est:P} of Theorem \ref{Uhlenbeck} hold with $\Omega_k$ in place of $\Omega$.

Then both $P_k$ and $\xi_k$ are uniformly bounded in $W^{2,2}$, and for any subsequence of $(\Omega_k)$ for which $P_k$ and $\xi_k$ are weakly convergent in $W^{2,2}$ to $P$ and $\xi$, respectively,  conditions \eqref{eq:uhlenbeck} and \eqref{est:P} of Uhlenbeck's Theorem do hold for $\Omega$, $P$ and $\xi$.
\end{corollary}

In other words, the decomposition for the limit matrix $\Omega$ can be effectuated with (any) weak limit of the transformations $P_k$ and forms $\xi_k$.

Indeed, assume $\Omega_k \rhu \Omega$ w $W^{1,2}$ and that $\Omega_k$ satisfy, uniformly with respect to $k$, assumptions of Theorem \ref{Uhlenbeck}. We obtain then, by the theorem, transformations $P_k$ and $\xi_k$ such that
\begin{equation}\label{eq:def P_k}
dP_k +P_k\Omega_k =\ast d\xi_k \, P_k \qquad \text{on }\bb^4
\end{equation}
and
\begin{equation}\label{eq:oszac P_k i xi_k}
\|\nabla^2 P_k\|_{L^2}+\|\nabla P_k\|_{L^4}+\|\nabla^2 \xi_k\|_{L^2}+\|\nabla \xi_k\|_{L^4}\les \|\nabla\Omega_k\|_{L^2}+\|\Omega_k\|^2_{L^4}\leq M,
\end{equation}
where $M>0$ is a constant which is independent of $k$.

Since, by \eqref{eq:oszac P_k i xi_k}, the sequences $(\xi_k)$ and $P_k$ are bounded in $W^{2,2}$, we can choose subsequences (for simplicity still indexed by $k$)  such that
\begin{align*}
\xi_k \rhu \xi \qquad &\text{weakly in $W^{2,2}$,}\\
P_k \rhu P \qquad &\text{weakly in $W^{2,2}$ and a.e.}
\end{align*}
Thus (after again choosing subsequence) we may assume that
\begin{align*}
d \xi_k \to d \xi \qquad &\text{strongly in $L^{4-\delta}$,}\\
d P_k \to d P \qquad &\text{strongly in $L^{4-\delta}$,} \\
\Omega_k \to \Omega \qquad &\text{strongly in $L^{4-\delta}$}.
\end{align*}
for any small $\delta>0$.
Since $P_k$ are also uniformly bounded in $L^\infty$, we can take the $L^{4-\delta}$-limit on both sides of \eqref{eq:def P_k} obtaining
\begin{equation}\
dP +P\Omega =\ast\, d\xi \,  \qquad \text{ on }\bb^4 
\end{equation}
in the sense of distributions, which proves that $P$, $\xi$ and $\Omega$ indeed satisfy \eqref{eq:uhlenbeck}. The remaining estimates and boundary conditions in Theorem \ref{Uhlenbeck} are obvious.

This sort of continuity of the decomposition of $\Omega$, i.e. the fact that the transformation of the limit $\Omega$ may be attained by taking the weak limits of the elements of decomposition of $\Omega_k$ allows us later to estimate the $W^{2,2}$-norm of differences between the elements of decomposition of $\Omega_k$ and $\Omega$.

\section{Proof}

Let us first give a rough sketch and plan of the proof. The key idea is to prove that $u$ solves the limiting system  \eqref{system} outside a countable set of points and then remove these possible singularities with the use of a properly chosen test function.  A standard argument, cf. Section~\ref{preparation}, shows that $\bb^4\setminus A_1$, where $A_1$ is finite, can be covered by balls $B_j$ such that $\|\Omega_k\|_{W^{1,2}(B_j)}$ is small, so that Uhlenbeck's decomposition can be applied inside each $B_j$ separately. Next, in Section~\ref{transformation}, we fix one of the $B_j$ and, following the crucial ideas of Lamm-Rivi\`{e}re \cite{LammRiviere} and Struwe \cite{Struwe}, we use the equation $(dP_k)P_k^{-1}+P_k\Omega_kP_k^{-1}=\ast d\xi_k$ on $B_j$ to rewrite \eqref{system_k} as
\begin{equation}\label{system_k-intro}
\begin{split}
\Delta( P_k \Delta u_k) + \text{a perturbation}
=H_k\cdot \nabla u_k
 + \ast d\Delta \xi_k\cdot (P_k \nabla u_k)\qquad\text{on \;} B_j,
\end{split}
\end{equation}
where $H_k$ depends on $D_k,E_k,G_k, P_k,\Omega_k$ and their derivatives.

It is a purely routine matter to pass to the weak limit on the left hand side of \eqref{system_k-intro}.  On the right hand side, after passing to subsequences, $H_k$ is bounded in $L^{4/3}$ and converges strongly in all $L^s$ for $s<\frac 43$, and $\nabla u_k$ is bounded in $L^4$ by the Sobolev imbedding, and converges strongly in $L^q$ for all $q<4$. Thus, $H_k\cdot\nabla u_k$ is bounded in $L^1$ but a priori does not have to converge in $\mathcal{D}'$. The second term, $\ast d\Delta \xi_k\cdot (P_k \nabla u_k)$, presents a similar difficulty. To cope with that, we apply in Section~\ref{convergence} P.-L. Lions' concentrated compactness method \cite{Lions}, following earlier work of Freire, M\"{u}ller and Struwe \cite{Freire} on wave maps and harmonic maps, and the second and third author \cite{SZG} on $H$-systems. The key idea is to exploit the existence of \emph{second} order derivatives of $u$. This yields
\[
H_k\cdot \nabla u_k
 + \ast d\Delta \xi_k\cdot (P_k \nabla u_k) \longrightarrow
H\cdot \nabla u
 + \ast d\Delta \xi\cdot (P \nabla u) + S_j \qquad\text{on \;} B_j,
\]
where $H\cdot \nabla u
 + \ast d\Delta \xi\cdot (P \nabla u)$ is the desired term of the limit system \eqref{system} rewritten in the $(P,\xi)$ gauge, and $S_j$ is a combination of Dirac delta measures, supported on a countable subset  $A$ of $B_j$.

To complete the proof, in Section~\ref{removing} we show that each $S_j$ must be zero, using a capacity argument, based on the fact that $W^{2,2}(\R^4)$ contains unbounded functions.
Thus, the limit $u$ of $u_k$ satisfies \eqref{system} in $\bb^4\setminus A_1$; another application of the same argument shows that $A_1$ must be empty.

\subsection{Preparation to Uhlenbeck's transformation}
\label{preparation}

Since the sequence $\Omega_k$ is weakly convergent, it is bounded in $W^{1,2}$; we shall denote the bound on its norm by $M$.

In what follows, we want to cover the ball $\bb^4$ by balls $B_j$ in such a way, that in every of the $B_j$ we may, after passing to a subsequence, assume that $\|\Omega_k\|_{W^{1,2}(B_j)}\leq \eps$, where $\eps$ is as in Corollary \ref{UhlCor2} to Uhlenbeck's Theorem. This might not be possible for the whole $\bb^4$, but it is outside a finite set of points. To visualize this better, replace $\bb^4$ by a four-dimensional cube and consider its dyadic decomposition into cubes $C_{i,j}$, where the second subscript $j$ counts subsequent cubes of a specified generation $i$.

A cube $C$ is {\it bad}, if one cannot choose a subsequence of $\Omega_k$ such that its $W^{1,2}$-norm is bounded on $C$ by $\eps$, i.e. if, for all $k$ sufficiently large, $\|\Omega_k\|_{W^{1,2}(C)}>\eps$. Notice that, in every generation of the dyadic decomposition, the number of {\it bad} cubes $C_{i,j}$ is bounded by the same constant $N=\lceil M \eps^{-1} \rceil$, and that if a cube is not {\it bad} (i.e. it is {\it good}), neither are all its descendants. The intersection of all {\it bad} cubes,
\begin{equation}
	\label{first bad}
	A_1:=\bigcap_{i=1}^\infty \biggl(\bigcup_{1\le j\le N} \text{bad cubes }C_{i,j}\biggr)
\end{equation}
is a finite set of (at most $N$) points, and any point not in $A_1$ lies in a cube that is {\it good}.

These points are, in fact, accumulation points of the $W^{1,2}$-norm of the weakly convergent sequence $(\Omega_k)$, more precisely --- these accumulation points, for which the energy loss exceeds~$\eps$.

For the next two subsections, let us fix an arbitrary {\it good} ball $B=B_j$ contained in $\bb^4\setminus A_1$. We consider (not changing the notation, to keep things simple), instead of the whole sequence $u_k$, only the subsequence of $u_k$ (and of $D_k$, $E_k$, $G_k$ and $\Omega_k$) for which $\|\Omega_k\|_{W^{1,2}(B_j)}\leq \eps$. By Corollary \ref{UhlCor2} we can assume that the Uhlenbeck's theorem holds for $B$, so there exist $P_k$ and $\xi_k$ such that $dP_k\,P_k^{-1}+P_k\Omega_k P_k^{-1}=\ast d\xi_k$.

\subsection{Transformation of the equation: calculations}
\label{transformation}

The calculations below follow closely and with more detail the brief calculations by Struwe in \cite{Struwe}. We provide them for the reader's convenience, and also because we shall need some more knowledge on the structure of certain terms in our reasoning.

We recall the indices of the multidimensional objects that appear in our system:
\begin{gather*}
u=(u^i), \qquad \nabla u=(\partial_\alpha u^i), \qquad
P=(P^{ij}), \qquad D=(D^{ij}_\alpha ),\\
E=(E^{ij}_{\alpha \beta }), \qquad G=G^{ij}_\alpha,
\qquad \Omega =(\Omega^{ij}_\alpha ),
\end{gather*}
with $i,j=1,\ldots,n$ and $\alpha ,\beta ={1,2,3,4}$. To simplify the notation without making the calculations ambiguous we shall use the standard summation conventions.

Furthermore, it is often convenient to omit at least some of the indices. In that case,
\begin{itemize}
\item multiplication of tensor objects that is denoted by a dot ($\cdot$) is a~standard scalar product in the $\R^4\times\R^m$ space, e.g.
\begin{equation*}
\begin{split}
E\cdot\nabla u&=\left(E^{ij}_{\alpha \beta }\partial_\beta u^j\right)^i_\alpha =\biggl(\sum_{\beta =1}^4\sum_{j=1}^n E^{ij}_{\alpha \beta }\partial_\beta u^j\biggr)^i_\alpha \\
G\cdot \nabla u&=\left(G^{ij}_\alpha \partial_\alpha u^j\right)^i=\biggl(\sum_{\alpha  =1}^4\sum_{j=1}^n G^{ij}_{\alpha}\partial_\alpha  u^j\biggr)^i
\end{split}
\end{equation*}
\item multiplication of tensor objects that is not denoted by any operator sign is standard matrix multiplication, e.g.
$$
P\Delta^2 u=\left(P^{ij}\Delta^2u^j\right)_i=\biggl(\sum_{j=1}^n P^{ij}\Delta^2 u^j\biggr)_i
$$
\item tensor multiplication ($\otimes$) denotes tensor product in $\R^4$ (and then, possibly, matrix multiplication in the coordinates), e.g.
\begin{equation*}
\begin{split}
\nabla P\otimes\nabla u&=\nabla P^{ij}\otimes \nabla u^j=\left( \left(\partial_\alpha P^{ij}\right)_\alpha \otimes \left(\partial_\beta u^j\right)_\beta\right)^i\\
&=\biggl(\sum_{j=1}^n \partial_\alpha P^{ij}\partial_\beta u^j\biggr)^i_{\alpha \beta }.
\end{split}
\end{equation*}
\end{itemize}

Below, we transform the system \eqref{system_k} for $u_k$. In the calculations that follow we omit the index $k$ (one should not confuse this temporary notational simplification with the claim that $u$ -- a weak limit of $(u_k)$ -- satisfies \eqref{system}; proving this is the goal of our paper).

Applying $P$ to the Laplacian on the left-hand side of the system, we obtain
\begin{equation*}
\begin{split}
\Delta (P\Delta u)^i
&=\Delta (P^{ij}\Delta u^j)=
\partial_\alpha \partial_\alpha (P^{ij}\partial_\beta \partial_\beta u^j)\\
&=P^{ij}\partial_\alpha\partial_\alpha\partial_\beta \partial_\beta u^j
+\partial_\beta \partial_\alpha \partial_\alpha P^{ij}\partial_\beta u^j
-\partial_\beta (\partial_\alpha \partial_\alpha P^{ij}\partial_\beta u^j)\\
&\quad +2\partial_\alpha \partial_\beta (\partial_\alpha P^{ij}\partial_\beta u^j)-2\partial_\alpha (\partial_\beta \partial_\alpha P^{ij} \partial_\beta u^j),
\end{split}
\end{equation*}
which can be rewritten shortly as
\begin{equation}\label{eq:star}
\begin{split}
\Delta (P\Delta u)^i=(P\Delta^2 u)^i &+ (\nabla (\Delta P)\cdot\nabla u)^i-\dv (\Delta P\nabla u)^i\\
&+2\dv^2(\nabla P\otimes \nabla u)^i-2\dv(\nabla^2P\cdot \nabla u)^i.
\end{split}
\end{equation}
By the equation \eqref{system_k} satisfied by $u_k$ we have, still omitting the index $k$ for the sake of simplicity,
$$
(P\Delta ^2 u)^i=P^{ij}\Delta (D\cdot \nabla u)^j+P^{ij}\dv(E\cdot \nabla u)^j+P^{ij}(G\cdot \nabla u)^j+P^{ij}(\Delta \Omega \cdot \nabla u)^j,
$$
which can be rewritten as
\begin{equation}\label{eq:twostar}
\begin{split}
P\Delta ^2 u=\Delta (PD\cdot \nabla u)
&+\dv\left((PE-2\nabla P D)\cdot \nabla u\right)\\
&+(\Delta PD+PG+P\Delta \Omega -\nabla P\cdot E)\cdot\nabla u.
\end{split}
\end{equation}
Substituting \eqref{eq:twostar} into \eqref{eq:star}, after some rearranging, yields
\begin{equation}
	\label{eq:threestar}
\begin{split}
\Delta (P\Delta u)=\Delta &(PD\cdot \nabla u)+2\dv^2(\nabla P\otimes \nabla u)\\
&+\dv\left((PE-2\nabla P D-2\nabla^2P\cdot \nabla u)\cdot \nabla u-\Delta P\nabla u\right)\\
&+(\Delta PD+PG+P\Delta \Omega +\nabla(\Delta P)-\nabla P\cdot E)\cdot \nabla u.
\end{split}
\end{equation}
Define
\begin{equation}\label{eq:faleczki}
\begin{split}
\tD_\alpha^{il}&=P^{ij}D^{jl}_\alpha \\
\tE^{il}_{\alpha \beta }&=P^{ij}E^{jl}_{\alpha \beta }-2\partial_\alpha P^{ij}D^{jl}_\beta -2\partial_\alpha \partial_\beta P^{il}-\delta_{\alpha \beta}\Delta P^{il}\\
H&=\Delta PD+PG+P\Delta \Omega +\nabla (\Delta P)-\nabla P\cdot E-\ast d\Delta \xi P,
\end{split}
\end{equation}
where, as before, the Roman lowercase indices run from $1$ to $n$, and the Greek indices from $1$ to $4$; $\delta_{\alpha \beta }$ denotes Kronecker's delta.

With this notation, we rewrite \eqref{eq:threestar} as
\begin{equation}
	\label{system_k_transformed}
	\begin{split}
\Delta (P\Delta u)=\Delta (\tD\cdot\nabla u)& +2\dv^2(\nabla P\otimes\nabla u)\\ &+\dv(\tE\cdot\nabla u)+H\cdot\nabla u+\ast d\Delta \xi \cdot P\nabla u.
\end{split}
\end{equation}
We shall need the precise form and integrability properties of terms that appear in $H$. By Theorem \ref{Uhlenbeck},
$$
\ast d\Delta \xi P=\Delta \left((\nabla P+P\Omega )P^{-1}\right)P,
$$
thus
$$
H=\Delta PD+PG+P\Delta \Omega +\nabla (\Delta P)-\nabla P\cdot E-\Delta \left((\nabla P+P\Omega )P^{-1}\right)P,
$$
and after some simple reductions we get
\begin{equation*}
\begin{split}
H^{ij}_\beta =\Delta P^{il}D^{lj}_\beta &+P^{il} G^{lj}_\beta-\partial_\alpha P^{il}E^{lj}_{\alpha \beta }-\Delta P^{il}\Omega^{lj}_\beta \\
&-2\Big(\partial_\alpha P^{il}\partial_\alpha \Omega^{lj}_\beta +\partial_\alpha \partial_\beta P^{il}\partial_\alpha (P^{-1})^{ls}P^{sj}\\
&{}\qquad+\partial_\alpha P^{il}\Omega^{lt}_\beta \partial_\alpha (P^{-1})^{ts}P^{sj}+P^{il}\partial_\alpha \Omega^{lt}_\beta \partial_\beta (P^{-1})^{ts}P^{sj}\Big)\\
&-\partial_\beta P^{il}\Delta (P^{-1})^{ls}P^{sj}-P^{il}\Omega ^{lt}_\beta \Delta (P^{-1})^{ts}P^{sj},
\end{split}
\end{equation*}
or, in the simplified notation, without the jungle of indices,
\begin{equation}\label{struct. of H}
\begin{split}
H=\Delta PD & +PG-\nabla P\cdot E-\Delta P\Omega\\
&-2\Bigl(\nabla P\cdot\nabla \Omega +(\nabla^2 P\cdot \nabla (P^{-1}))P\Bigr.\\ &\Bigl.{}\qquad+
 (\nabla P\Omega )\cdot\nabla (P^{-1})P+P(\nabla \Omega \cdot\nabla (P^{-1}))P\Bigr)\\
&-\nabla P\Delta (P^{-1})P-P\Omega \Delta (P^{-1})P.
\end{split}
\end{equation}
Please bear in mind that in fact we shall use equations \eqref{eq:faleczki}--\eqref{struct. of H} for each $k$, adding the subscript $k$ to all letters $H,D,E,G,P,\Omega,\xi,\tilde D, \tilde E$.

From now on we shall return to using the index $k$, where appropriate.

The bounds for $D_k$, $E_k$, $G_k$ and $\Omega_k$ together with the estimates for $P_k$ and $\xi_k$ given by Theorem \ref{Uhlenbeck} imply that $H_k\in L^{4/3}$ and
\begin{equation}\label{est. on H_k}
\|H_k\|_{L^{4/3}}\le C=C\Big(\sup_k\max\big(\|D_k\|_{W^{1,2}},\|E_k\|_{L^2},\|G_k\|_{L^{4/3}},\|\Omega_k\|_{W^{1,2}}\big)\Big)\, .
\end{equation}
To check this, one just uses H\"older's inequality, and -- when it is appropriate -- the Sobolev imbedding $W^{1,2}\subset L^4$ in dimension $4$; we leave the details to the reader.
\begin{remark*} If  the coefficients $D_k=D(u_k,\nabla u_k)$, $E_k=E(u_k,\nabla u_k,\nabla^2 u_k)$, $G_k=G(u_k,\nabla u_k)$ and $\Omega_k=\Omega(u_k,\nabla u_k)$ are given by the composition of fixed smooth functions with the~$u_k$ and their derivatives, and satisfy the growth conditions \eqref{Gr E}, then a computation yields
	\[
	\|H_k\|_{L^{4/3}}\lesssim \|u_k\|_{W^{2,2}}^{3/2} +\|\nabla u_k\|_{L^4}^3\, .
	\]
We do not rely on that particular estimate, though.
\end{remark*}

Let  $H$ be defined analogously to $H_k$, i.e.\ by formula \eqref{struct. of H}.\footnote{This time the subscript $k$ in \eqref{struct. of H} is \emph{really} omitted, not just for the sake of simplicity!} The convergence of all the terms in the right hand side of the formula for $H_k$ is such that, up to a subsequence, we can assume
\begin{equation}
	\label{Hk strong}
	H_k\rhu H\quad \text{ in }L^q\text{ for all }1\le q< \frac{4}{3}.
\end{equation}
To see this, we just use the elementary observation: \emph{if $\frac 1p + \frac 1q=\frac 1r$, $p,q,r>1$, and $f_k\to f$ in $L^p$ and $g_k\rhu g$ in $L^q$, then $f_kg_k\rhu fg$ in $L^r$,} combining it with the imbedding $W^{2,2}\subset W^{1,4}$ and Rellich-Kondrashov's compactness theorem.

Moreover, by the estimates on $H_k$, we can once again choose a~subsequence of $H_k$ that is convergent weakly in $L^{4/3}$.

We write out the weak formulation of \eqref{system_k} in $B_j$, using its transformed form \eqref{system_k_transformed}, and separating the terms into `easy' (left hand side) and `hard' (right hand side): the identity
\begin{equation}\label{weaksystem_k}
\begin{split}
\int P_k \Delta u_k\Delta \psi
&-\int (\tD_k\cdot \nabla u_k)\Delta \psi \\
&+ \int \Big[2\dv(\nabla P_k\otimes\nabla u_k)+(\tE_k\cdot \nabla u_k)\Big]
\cdot \nabla \psi \\
&=\int(H_k\cdot \nabla u_k)\psi
+\int\ast d\Delta \xi_k\cdot (P_k \nabla u_k)\psi
\end{split}
\end{equation}
holds for each smooth map $\psi$ compactly supported in $B=B_j$. (Since $\xi_k$ is only of class $W^{2,2}$, the last term has to be interpreted using one integration by parts.)

\subsection{Convergence of \eqref{weaksystem_k}}
\label{convergence}

We consider the left hand side and the right hand side separately. (The key difficulty is to prove that right hand sides converge to the appropriate limit).

\subsubsection{Convergence of the left hand side of \eqref{weaksystem_k}}
By assumption, $u_k\rhu u$ in $W^{2,2}$, thus, after passing to a subsequence, we may assume that
\begin{equation*}
\begin{split}
u_k\to u \quad &\text{ in }L^s \text{ for all }1\leq s<\infty,\\
\nabla u_k\rhu \nabla u \quad &\text{ in }L^4,\\
\nabla u_k\to \nabla u \quad &\text{ in }L^p \text{ for all }1\leq p<4,\\
\nabla^2 u_k\rhu \nabla^2 u \quad &\text{ in }L^2.
\end{split}
\end{equation*}

Taking into account the bounds for $D_k$, $E_k$, $G_k$ and $\Omega_k$ we can also assume that (again, after passing to a subsequence) there exist $D$, $E$, $G$ and $\Omega $ such that
\begin{gather*}
D_k\rhu D \quad\text{and}
\quad \Omega_k\rhu \Omega\quad \text{ in }W^{1,2},\\
E_k\rhu E \quad \mbox{in $L^p$ for all $1\le p< 2$},\qquad
 G_k\rhu G \quad \mbox{in $L^s$ for all $1\le s< 4/3$}.
\end{gather*}

Similarly, Theorem \ref{Uhlenbeck} gives uniform estimates on $P_k$ and $\xi_k$, which allow us to assume (after passing to subsequences) that
\begin{equation*}
\begin{split}
(P_k,\xi_k)\rhu (P,\xi) \quad &\text{ in }W^{2,2},\\
\end{split}
\end{equation*}
and, by the Rellich--Kondrashov's compactness theorem, $\nabla P_k\to\nabla P$ and $\nabla \xi_k\to \nabla \xi$ strongly in $L^s$ for all $1\le s< 4$, and $P_k\to P$ in $L^s$ for all $1\le s< \infty$.
The limits $P$ and $\xi $ of $P_k$ and $\xi_k$ satisfy the claim of Corollary~\ref{cor:2} on~$B=B_j$. In particular,
$$
dP P^{-1}+P\Omega P^{-1}=\ast d\xi.
$$
By H\"older's inequality, it follows from all these convergence assumptions that $\nabla P_k\otimes\nabla u_k\to \nabla P\otimes \nabla u $ in $\mathcal{D}'$,
\[\tD_k\cdot \nabla u_k=P_kD_k\cdot \nabla u_k\to PD\cdot \nabla u =\tD\cdot \nabla u \qquad\mbox{in $\mathcal{D}'$},
\]
and finally,
\[
\tE_k\cdot \nabla u_k\to \tE\cdot \nabla u
\qquad\mbox{in $\mathcal{D}'$}
\]
(to check this, see the \eqref{eq:faleczki} for relation between $\tE$ and $D, E,$ and $P$).

Thus, up to passing to a subsequence, for each fixed test map $\psi$ the left hand side of \eqref{weaksystem_k} converges to
$$
\int P\Delta u \Delta \psi -\int (\tD\cdot \nabla u)\Delta \psi +\int\left[ 2\dv(\nabla P\otimes\nabla u)+\tE\cdot\nabla u\right]\cdot\nabla \psi .
$$

\subsubsection{Convergence of the right hand side of \eqref{weaksystem_k}}
Now let us concentrate on the convergence of the right hand side of \eqref{weaksystem_k}, i.e. of
\begin{equation}\label{rhs of weaksystem}
\int (H_k\cdot\nabla u_k)\psi +\int \ast d\Delta \xi_k\cdot (P_k \nabla u_k)\psi.
\end{equation}
This is the heart of the matter. A priori, by H\"{o}lder's inequality, $H_k\cdot \nabla u_k$ is bounded just in $L^1$, and there is no simple means of passing to the limit, as \emph{neither} $H_k\to H$ in $L^{4/3}$, \emph{nor} $\nabla u_k\to \nabla u$ in $L^4$; the convergence in both cases is weak. The second term presents a similar problem.
To circumvent this difficulty,  we shall study the convergence of \eqref{rhs of weaksystem} using the concentration--compactness method of P.-L.~Lions. To deal with the first term in \eqref{rhs of weaksystem}, we exploit the fact that $u_k$ has \emph{second} order derivatives and use Sobolev inequality. To cope with the second term, we proceed in a similar way, employing also the equations satisfied by $P$ and $\xi$, and integration by parts.

We define the auxiliary distributions $T_k$ by
\begin{equation}\label{def T_k}
\left\langle T_k,\psi \right\rangle =\int (H_k\cdot\nabla u_k)\psi +\int \ast d\Delta \xi_k\cdot (P_k \nabla u_k)\psi, \quad  \psi\in \cC_o^\infty(B,\R^n).
\end{equation}
Coordinatewise, we write $T_k=(T_k^1,\ldots,T_k^n)$, with $T_k^i\in \mathcal{D}'(\R^n)$ given by
\begin{equation*}
\left\langle T_k^i,\phi \right\rangle =\int (H^i_k\cdot\nabla u_k)\phi +\int \ast d\Delta \xi^i_k\cdot (P_k \nabla u_k)\phi, \quad  \phi\in \cC_o^\infty(B).
\end{equation*}
We define the distribution $T$ by a formula analogous to \eqref{def T_k}, omitting the index $k$ everywhere.

Let us recall again that, for $i=1,\ldots,n$,
\begin{eqnarray*}
H^{i}&=&\left(H^{ij}_\alpha \right),\\
\ast d\xi^i&=&\left(\ast d\xi\right)^{ij}_\alpha,
\end{eqnarray*}
with $j=1,\ldots,n$ and $\alpha =1,\ldots,4$.
In view of that we have
\begin{eqnarray*}
\phi H^i\cdot\nabla u&=&\phi H^{ij}_\alpha \partial_\alpha u^j\\
\phi \; \ast d\Delta \xi^i\cdot P\nabla u&=&\phi \left({}\ast d \Delta \xi \right)^{ij}_\alpha P^{jk}\partial_\alpha u^l.
\end{eqnarray*}
\begin{lemma}\label{prawastr}
There exists a subsequence $T_{k'}$ out of $T_k$ that
$$
T^i_{k'}\to T^i+\sum_{l\in J}a_{li}\delta_{x_{li}}\quad \text{ in }\mathcal{D}'(\R^n), \qquad i=1,\ldots, n,
$$
where $J$ is (at most) countable, $a_{li}\in \R$, $x_{li}\in \bb^4$ and $\sum_{l\in J}|a_{li}|^{4/5}<\infty$ for $i=1,\ldots, n$.
\end{lemma}
\begin{proof}
Our aim is to prove that there exist non-negative, uniformly bounded Borel measures $\mu_k$ such that for any $\phi \in \cC_o^\infty(B)$ and $i=1,\ldots,n$ we have
\begin{equation}
\left| \left\langle T_k^i-T^i,\phi^5\right\rangle\right|\lesssim \left(\int \phi^{ 4} \, d\mu_k\right)^{5/4}+o(1)\qquad\mbox{as}\quad k\to\infty.
\end{equation}
This will allow us to use the method of concentrated compactness.

We have
\begin{equation}
\begin{split}
\left\langle T^i_k-T^i,\phi ^5\right\rangle &=\int \phi^5 \left(H^i_k\cdot\nabla u_k-H^i\cdot\nabla u\right)\\&{}\qquad+\int\phi^5\left(\ast d \Delta \xi^i_k\cdot P_k \nabla u_k-\ast d\Delta \xi^i\cdot P\nabla u\right)\\
&=A+B.
\end{split}
\end{equation}
Each of the integrals shall be dealt with separately.

\medskip
\noindent
\textbf{Estimate of $A$:} we split this integral into two,
$$
A=\int \phi^5 H^i_k\cdot (\nabla u_k-\nabla u)+\int \phi^5(H^i_k-H^i)\cdot \nabla u=\mathit{I}+\mathit{II}.
$$
For $\mathit{I}$ we have
$$
\mathit{I}=\int \phi^3 H^i_k\cdot \nabla \left(\phi^2(u_k-u)\right)-2\int \phi^4H^i_k\cdot \left(\nabla \phi \otimes(u_k-u)\right).
$$
We have assumed that $u_k\to u$ in $L^s$ for $1\leq s<\infty$ and that $\nabla u_k\to \nabla u$ in $L^p$ for $1\leq p<4$, moreover, we know that the $H_k$ are uniformly bounded in $L^{4/3}$. By H\"older's and Sobolev's inequalities
\begin{equation}
\begin{split}
|\mathit{I}|&\leq \int |\phi |^3|H_k||\nabla (\phi^2(u_k-u))|+o(1)\\
&\leq \left(\int \phi^4 |H_k|^{4/3}\right)^{3/4}\left(\int\left|\nabla \left(\phi^2(u_k-u)\right)\right|^4\right)^{1/4}+o(1)\\
&\lesssim \left(\int \phi^4 |H_k|^{4/3}\right)^{3/4}\left(\int\left|\nabla^2 \left(\phi^2(u_k-u)\right)\right|^2\right)^{1/2}+o(1)\\
&\lesssim \left(\int \phi^4 |H_k|^{4/3}\right)^{3/4}\left(\int\phi^4 \left|\nabla^2 (u_k-u)\right|^2\right)^{1/2}+o(1)\\
&\lesssim \left(\int \phi^4 \left(|H_k|^{4/3}+|\nabla^2u_k|^2+|\nabla ^2u|^2\right)\right)^{5/4}+o(1).
\end{split}
\end{equation}
With $\mathit{II}$ we proceed in a similar way:
\begin{equation}
\begin{split}
\mathit{II}&=\int \phi^3(H^i_k-H^i)\cdot \nabla (\phi^2u)-2\int \phi^4(H^i_k-H^i)\cdot (\nabla \phi \otimes u)\\
&=\mathit{II_a}+\mathit{II_b}.
\end{split}
\end{equation}
We estimate $\mathit{II_a}$ in the same way as we did with $I$:
$$
|\mathit{II_a}|\lesssim \left(\int \phi^4\left(|H_k|^{4/3}+|H|^{4/3}+|\nabla^2 u|^2\right)\right)^{5/4}+o(1).
$$
The integral $\mathit{II_b}$ converges (up to choosing a subsequence) to zero, since $u\in L^s$ for any $s\geq 1$ and $H_k\rhu H$ in $L^q$ for $q<\frac{4}{3}$.

\medskip
\noindent
\textbf{Estimate of $B$:} this integral is the sum of three terms,
\begin{equation}
\begin{split}
B=&\int\phi^5\left(\ast d \Delta \xi^i_k\cdot P_k \nabla u_k-\ast d\Delta \xi^i\cdot P\nabla u\right)\\
&=\int\phi^5\ast d \Delta \xi^i_k\cdot P_k \nabla (u_k-u)+\int\phi^5\ast d \Delta \xi^i_k\cdot (P_k-P) \nabla u\\
&\qquad +\int\phi^5\left(\ast d \Delta \xi^i_k-\ast d \Delta \xi^i\right)\cdot P\nabla u\\
&=\mathit{III}+\mathit{IV}+\mathit{V}.
\end{split}
\end{equation}
Since, thanks to the boundary conditions on $\Delta \xi^i_k$,
$$
\int\ast d \Delta \xi^i_k\cdot \nabla \left( \phi^5 P_k (u_k-u)\right)=0,
$$
we obtain
\begin{equation}
\begin{split}
\mathit{III}&=-5\int \phi^4\ast d\Delta \xi^i_k\cdot\left(\nabla \phi \otimes P_k(u_k-u)\right)\\&{}\qquad -\int \phi^5\ast d\Delta \xi ^i_k\cdot(\nabla P_k)(u_k-u)\\
&=\mathit{III_a}+\mathit{III_b}.
\end{split}
\end{equation}
Integrating by parts, taking into account the convergence of $u_k$ to $u$ in any $L^s$ and arbitrary integrability of $P_k$ we obtain
\begin{equation}
\begin{split}
|\mathit{III_a}|&=\left|\int \Delta \xi^i_k\wedge d\left(\phi^4P_k(u_k-u)\right)\wedge d\phi\right|\\
&\lesssim \int \phi^4|\nabla\phi ||\nabla^2\xi_k||\nabla P_k||u_k-u|+\int \phi^4|\nabla \phi ||\nabla^2\xi_k||P_k||\nabla u_k-\nabla u|\\
&=o(1).
\end{split}
\end{equation}
For $\mathit{III_b}$ we get, integrating by parts,
\begin{equation}
\begin{split}
|\mathit{III_b}|&=\left|\int \Delta \xi^i_k d P_k \wedge d\left(\phi^5 (u_k-u)\right)\right|\\
&\leq \left|\int \phi^5 \Delta \xi^i_k d P_k \wedge d(u_k-u)\right|+\int \phi^4|\nabla \phi ||\nabla^2\xi _k||\nabla P_k||u_k-u|\\
&\leq \left|\int \left(\phi^2 \Delta \xi^i_k\right)\left(\phi\,  dP_k\right)\wedge\left(\phi^2\, d(u_k-u)\right)\right|+o(1)\\
&\leq \left|\int \left(\phi^2 \Delta \xi^i_k\right)\left(\phi\,  dP_k\right)\wedge d\left(\phi^2 (u_k-u)\right)\right|+o(1).
\end{split}
\end{equation}
As before, using H\"older's and Sobolev's inequalities, we estimate the last integral by the product
\[
\left(\int \phi^4|\nabla^2\xi_k|^2\right)^{1/2}\left(\int \phi^4|\nabla P_k|^4\right)^{1/4}\left(\int |\nabla^2(\phi^2(u_k-u))|^2\right)^{1/2}\, .
\]
This yields
\begin{equation}
\begin{split}
	\label{est:III_b}
\mathit{III_b} &\lesssim \left(\int \phi^4\left(|\nabla^2\xi_k|^2+|\nabla P_k|^4+|\nabla^2 u_k|^2+|\nabla^2 u|^2\right)\right)^{5/4}+o(1).
\end{split}
\end{equation}
The remaining integrals in $B$ -- the terms $\mathit{IV}$ and $\mathit{V}$ -- are estimated in much the same way. Integrating by parts, and then dealing as in the proof of \eqref{est:III_b}, we obtain
\begin{equation}\label{eq: IV}
\begin{split}
|\mathit{IV}|&=\left|\int \phi^5\ast d\Delta \xi^i_k\cdot (P_k-P)\nabla u\right|\\
&=\left|\int\Delta \xi^i_k d \left(\phi^5 (P_k-P)\right))\wedge d u\right|\\
&\leq \left|\int\phi^2 \Delta \xi^i_k\; d\big(\phi (P_k-P)\big)\wedge d (\phi^2 u)\right|\\& \qquad +4\int\phi^4|\nabla \phi ||\nabla^2\xi_k||P_k-P||\nabla u|+o(1)\\
&\lesssim \left(\int \phi^4\left(|\nabla^2\xi_k|^2+|\nabla P_k|^4+|\nabla P|^4+|\nabla^2 u|^2\right)\right)^{5/4}+o(1).
\end{split}
\end{equation}
Similarly,
\begin{equation}\label{eq: V}
\begin{split}
|\mathit{V}|&=\left|\int\phi^5\ast d\Delta (\xi^i_k-\xi^i)\cdot P\nabla u\right|\\
&=\left|\int \Delta (\xi^i_k-\xi^i)d(\phi^5 P)\wedge du\right|\\
&\leq \left|\int\phi^5 \Delta (\xi^i_k-\xi^i)d P\wedge du\right|+\left|\int\phi^4 \Delta (\xi^i_k-\xi^i)\wedge d\phi \wedge P du\right|\\
&\lesssim \left(\int \phi ^4\left(|\nabla^2\xi_k|^2+|\nabla^2 \xi |^2+|\nabla P|^4+|\nabla^2 u|^2\right)\right)^{5/4}+o(1).
\end{split}
\end{equation}
(in the calculations above we use H\"older's and Sobolev's inequalities together with the fact that $\nabla^2\xi_k\rhu\nabla^2\xi $ in $L^2$).

Altogether we obtain the estimate
$$
|B|\lesssim \left(\int \phi ^4\, f_k\; \right)^{5/4}+o(1),
$$
where $f_k:=|\nabla^2\xi_k|^2+|\nabla^2 \xi |^2+ |\nabla P_k|^4+|\nabla P|^4+|\nabla^2 u_k|^2+|\nabla^2 u|^2\in L^1$.
Putting together the estimates for $A$ and $B$ we get
\begin{equation}
\left|\left\langle T^i_k-T^i,\phi^5\right\rangle \right|\lesssim \left(\int \phi^4 d\mu_k\right)^{5/4}+o(1),
\end{equation}
with
\begin{equation}
\mu_k=|H_k|^{4/3}+|H|^{4/3}+f_k.
\end{equation}
Note that by our assumptions the $\mu_k$ are uniformly bounded in $L^1$. Passing to the limit in the space of measures we obtain
$$
T^i_k-T^i\to d\nu ,\qquad d\mu_k \to d\mu,
$$
and
$$
\left|\int \phi^5 d\nu \right|\lesssim \left(\int \phi^4 d\mu \right)^{5/4}.
$$
Now the claim of Lemma~\ref{prawastr} follows directly from the concentration-compactness lemma of P.-L. Lions, cf. \cite[Lemma 1.2, p.~161]{Lions}.
\end{proof}

Passing to the limit in \eqref{weaksystem_k}, we obtain
\begin{equation}\label{lim in Bj}
\begin{split}
\int P \Delta u\Delta \psi
&-\int (\tD\cdot \nabla u)\Delta \psi \\
&+ \int \Big[2\dv(\nabla P\otimes\nabla u)+(\tE \cdot \nabla u)\Big]
\cdot \nabla \psi \\
&=\int(H\cdot \nabla u)\psi
+\int\ast d\Delta \xi\cdot (P \nabla u)\psi +\langle S_j,\psi\rangle
\end{split}
\end{equation}
for each smooth test map $\psi$ compactly supported in $B_j$, with $H$, $\tD$ and $\tE$ given by \eqref{eq:faleczki}, and $(P,\xi)$ related to $\Omega=\lim \Omega_k$ via Uhlenbeck's Theorem~\ref{Uhlenbeck}. The singular distribution $S_j\in \mathcal{D}'$ is a countable series of Dirac delta measures with vector-valued coefficients, $S_j=\sum_\ell a_\ell \delta_{x_\ell}$. It follows from Lemma~\ref{prawastr} that the series $\sum |a_\ell|_1$ converges.

\subsection{Removing the singularities}
\label{removing}

To show that the limit $u$ of $u_k$ satisfies the limiting system not just in $B_j\setminus A$, where $A=\{x_\ell\colon \ell\in J\}$ is countable, but in fact in the whole $B_j$, we rely on the fact that $W^{2,2}(\R^4)$ contains unbounded functions. More precisely, the following holds.

\begin{lemma} There exists a sequence of functions $\Phi_k\in C_o^\infty (\R^4)$ such that $0\le \Phi_k\le 1$, $\Phi_k\equiv 1$ on $B(0,r_k)$ for some $r_k>0$,
	\begin{equation}
		\label{placek}
		\Phi_k\equiv 0 \quad\text{on } \R^n\setminus B(0,R_k), \qquad R_k\to 0 \quad\text{as \;} k\to \infty,
	\end{equation}
and \begin{equation}
\label{Phi k w22}
	\|\Phi_k\|_{W^{2,2}}\to 0 \qquad\mbox{as } k\to\infty.
\end{equation}
\label{capacity}	
\end{lemma}

For the reader's convenience, a proof of this lemma is sketched in the appendix.

\smallskip

Fix $\ell_0\in J$ and assume for the sake of simplicity that $x_{\ell_0}=0\in B_j$. Testing each equation $i=1,2,\ldots,n$ of system \eqref{lim in Bj} with $\psi_k=(\pm\Phi_k, \ldots,\pm\Phi_k)$, the signs being equal to the signs of coordinates of the coefficient $a_{\ell_0}$ in $S_j$, and keeping in mind that
\[
P \Delta u, \; \tD\cdot \nabla u, \; \nabla P\otimes\nabla u,\; \Delta\xi \in L^2, \qquad H\cdot \nabla u\in L^1, \qquad d\big((P\nabla u)\psi_k\big) \in L^2,
\]
we easily obtain
\[
o(1)=o(1)+\big\langle S_j,\psi_k\rangle= o(1)+ |a_{\ell_0}|_1 + \sum_{\ell\not=\ell_0} \langle a_\ell, \psi_k(x_l) \rangle\, , \quad k\to\infty.
\]
By \eqref{placek} and Lebesgue's dominated convergence theorem, the sum $\sum_{\ell\not=\ell_0}$ above tends to $0$ as $k\to\infty$. Thus, upon passing to the limit, we obtain $a_{\ell}=0$ for $\ell=\ell_0$. It follows that $S_j=0$.

\medskip
Thus, relying on the above and selecting, via the standard diagonal procedure, a subsequence $k'\to\infty$ such that $u_k'\to u$ (and all the coefficient functions converge in their appropriate spaces) in all good balls $B_j$ simultaneously, we check that
\begin{equation}\label{lim off A1}
\Delta^2 u=\Delta (D\cdot\nabla u)+\dv (E\cdot \nabla u)+G\cdot\nabla u + \Delta \Omega\cdot \nabla u\qquad\text{in \;} \bb^4\setminus A_1,
\end{equation}
where $A_1=\{x_1,\ldots, x_s\}$, $s\le N$, denotes the finite set of bad points, cf. \eqref{first bad}. To see that in fact \eqref{lim off A1} holds in the whole ball $B^4$, we pick an arbitrary test function $\varphi\in C_0^\infty(\bb^4)$ and write it as the sum
\[
\varphi(x)=\sum_{j=1}^s \varphi(x)\Phi_k(x-x_j) + \varphi(x)\biggl(1-\sum_{j=1}^s\Phi_k(x-x_j)\biggr) =: \varphi_{0,k}(x)+\varphi_{1,k}(x).
\]
Then, $\varphi_{1,k}\to\varphi$ in $W^{2,2}$ as $k\to \infty$. Moreover, $\text{supp}\, \varphi_{1,k}\Subset \bb^4\setminus A_1$. It follows easily that the weak form of \eqref{lim off A1} holds not just for test functions $\varphi_{1,k}\in C_0^\infty (\bb^4\setminus A_1)$, but also for an arbitrary $\varphi\in C_0^\infty (\bb^4)$.

The whole proof of Theorem~\ref{main theorem} is complete now.

\section{An example}
\label{example}

The following example shows that a system of type \eqref{system} satisfying assumptions \eqref{c spaces} may have solutions that are in $C^1 \setminus C^2$.

\medskip

Set $\phi(r)=\log\log(e-\log r)$ for $r>0$.

\begin{lemma}
\label{unbounded w22}
The function $f(x)=\phi(|x|)=\log (\log (e-\log (|x|)))$ is in $W^{2,2}(\bb^4)$.
\end{lemma}
\begin{proof}
The function is obviously in $L^2(\bb^4)$, therefore (e.g. by Gagliardo-Nirenberg's inequality) it is enough to check that $\nabla^2 f\in L^2(\bb^4)$. This is done by an elementary computation which, for the sake of completeness, is sketched in the appendix.
\end{proof}

\begin{lemma}\label{third} The functions $t\phi'(t)$, $t^2\phi''(t)$ and $t^3\phi'''(t)$ are bounded on $(0,1]$.
\end{lemma}

\begin{proof} A computation shows that
\begin{eqnarray*}
|t\phi'(t)|&=&\frac{1}{(e-\log t)\log(e-\log t)}\; ,\\
|t^2 \phi''(t)|&=&\frac{-1+(e-1-\log t)\log(e-\log t)}{(e-\log t)^2\log^2(e-\log t)}\; ,\\
|t^3 \phi'''(t)|&=&\frac{-2 + 3\,\left( -1 + e - \log (t) \right) \,
     \log (e - \log (t))}{{\left( e -
        \log (t) \right) }^3\,{\log (e - \log (t))}^3
    }\\
 &&-\frac{
    \left( 2 - 3\,e + 2\,e^2 +
       \left( 3 - 4\,e \right) \,\log (t) +
       2\,{\log (t)}^2 \right) \,
     {\log (e - \log (t))}^2}{{\left( e -
        \log (t) \right) }^3\,{\log (e - \log (t))}^3
    }.
\end{eqnarray*}
Clearly, the right hand sides converge to $0$ as $t\to 0$. The lemma follows.
\end{proof}

\begin{example}
Define $w:\bb^4\to \R$,
\begin{equation}
	\label{wu} w(x)=|x|^2\sin \phi (|x|).
\end{equation}
One easily checks, using Lemma~\ref{third}, that $\nabla w$ and $D^2 w$ are bounded on $\bb^4$. However,
$$
\frac{\partial^2w}{\partial x_1^2}(0,0,0,t)=2\sin\phi (t)+3t\phi'(t)\cos\phi (t)
$$
does not have a limit for $t\to 0$, and thus $w$ \emph{is not} a function of class $C^2$. Finally,
$$
|\nabla^3 w(x)|\lesssim \frac{1}{|x|(e-\log |x|)\log(e-\log |x|)}\in L^4(\bb^4).
$$

Set $C>0$ to be the bound on $|\nabla w|$ on $\bb^4$ and consider \begin{equation}
\label{v a}
v(x)=w(x)+2C\sum_{i=1}^4 x_i\, , \qquad a(x)=\Delta v|\nabla v|^{-2}\, .	
\end{equation}
By definition, $|\nabla v|\geq C>0$ on $\bb^4$, and
$|\nabla v|$ and $|D^2 v|$ are bounded on $\bb^4$. Thus,  $a(x)$ is bounded on $\bb^4$.

We also have $|D^3 v|=|D^3 w|\in L^4(\bb^4)$, and
$$
|\nabla a|\lesssim \frac{|\nabla \Delta v||\nabla v|^2+|\Delta v||\nabla v||D^2 v|}{|\nabla v|^4}\lesssim 1+|D^3 v|\in L^4(\bb^4).
$$
Thus, $a(x)\in W^{1,4}(\bb^4)\cap L^\infty(\bb^4)$. Consider the equation
\begin{equation}\label{ex:equation}
\Delta(\Delta u(x))=\Delta\left(a(x)|\nabla u|^2\right)
\end{equation}
Clearly this is an equation of the type
\begin{equation}\label{ex:eq type}
\Delta^2 u=\Delta\big( D(u,\nabla u)\cdot \nabla u\big) +\text{other terms (with zero coefficients)},
\end{equation}
where $D=a(x)\nabla u$ satisfies
\begin{itemize}
\item $D\in W^{1,2}$ whenever $u\in W^{2,2}$,
\item $|D(x)|\leq \|a\|_{L^\infty} |\nabla u(x)|$,
\item $\|\nabla D\|_{L^2}\leq \|\nabla a\|_{L^4} \|\nabla u\|_{L^4}+\|a\|_{L^\infty}\|D^2 u\|_{L^2}$.
\end{itemize}
It is also obvious, by the very definition of $a$, that $v$ solves \eqref{ex:equation}. Therefore
\eqref{ex:eq type}, under the conditions on $D$ listed above, admits non-smooth solutions.

Note, however, that the function $D$ given above \emph{does not} satisfy the pointwise estimate for $\nabla D$ given in \eqref{Gr E}.
\end{example}

\appendix

\section*{Auxiliary lemmata}
\setcounter{section}{1}
\setcounter{equation}{0}

\subsection*{An unbounded function in $W^{2,2}(\bb^4)$.} As explained in Section~4, to show that $f(x)=\log\big(\log(e-\log|x|)\big)$ is in $W^{2,2}$ on the unit ball $\bb^4$, it is enough to check that $D^2f\in L^2$. We have (Mathematica) for $x\in \bb^4$:
\begin{equation*}
\begin{split}
\left|\partial^2_{x_i x_j} f(x)\right|&=\left|-\frac{x_i x_j\left(1+(1-2e+2\log |x|)\log(e-\log|x|)\right)}{|x|^4(e-\log|x|)^2\log^2(e-\log|x|)}\right|\\
&\lesssim \left|\frac{x_i x_j}{|x|^4(e-\log|x|) \log(e-\log|x|)}\right|\\
&\lesssim \frac{1}{|x|^2\big(\!-\log |x|\big)\log\big(\!-\log |x|\big)}.
\end{split}
\end{equation*}
Integrating the above squared over $\bb^4$, in polar coordinates, amounts to calculating
$$
\int_0^1 \frac{dr}{r \log^2 r \log^2(-\log r)}=\int_1^\infty\frac{dt}{t \log^2 t \log^2 \log t},
$$
with the latter integral convergent, since the series
$$
\sum \frac{1}{n\log^2 n\log^2 \log n}
$$
is convergent, by Cauchy's condensation test.

\subsection*{Proof of Lemma~\ref{capacity}} Let $f\in W^{2,2}$ be the function introduced in Lemma~\ref{unbounded w22}. Fix a function $\zeta\in C^\infty(\R)$ with $0\le \zeta\le 1$, $\zeta\equiv 0 $ on $[1, +\infty)$ and $\zeta\equiv 1$ on $(-\infty, \frac 14)$. Set
\begin{equation}
	\Phi_k(x) =\zeta(k+1-f(x)), \qquad x\in \bb^4, \quad k=1,2,\ldots
\end{equation}
Clearly, $\Phi\in W^{2,2}(\bb^4)$. Moreover, $\Phi_k(x)=0$ if $f(x)\le k$. The equality $\Phi_k(x)\equiv 1$ holds  whenever $f(x)\ge k+\frac 34$, i.e. on a neighborhood of $0$. Since $f\ge 0$ on $\bb^4$ and $f$ is smooth except at $0$, $\Phi_k$ is smooth on $\bb^4$, and can be extended by $0$ to the whole space $\R^4$.
 Finally,
\begin{gather*}
|\nabla \Phi_k(x)|\lesssim |\nabla f(x)|\,\cdot\, \charfn_{\{f>k\}},\\
|\nabla^2\Phi_k(x)| \lesssim \big(|\nabla^2 f(x)| + |\nabla f(x)|^2\big)\,\cdot\, \charfn_{\{f>k\}}\, .
\end{gather*}
Since $f\in W^{2,2}\cap W^{1,4}$ is radial and $f\to +\infty$ at $0$, we obtain $\|\Phi_k\|_{W^{2,2}}\to 0$ as $k\to \infty$. For $x\not=0$, we simply have $\Phi_k(x)=0$ for all $k$ sufficiently large. This completes the proof.

\bibliography{bihar-compactness}{}
\bibliographystyle{amsplain}

\end{document}